\documentclass[11pt,letterpaper,reqno]{amsart}

\usepackage{tikz}
\usetikzlibrary{positioning, shapes.geometric, arrows.meta}
\usepackage{amssymb}
\usepackage{amsmath}
\usepackage{amsthm}
\usepackage{amsfonts}
\usepackage{bbm}
\usepackage{enumitem} 
\usepackage{pgfplots}
\pgfplotsset{compat=1.18} 
\usepackage{booktabs}

\usepackage{graphicx}
\usepackage[T1]{fontenc}
\usepackage{doi}
\usepackage{float} 
\usepackage{bookmark}
\usepackage{hyperref}
\hypersetup{pdfstartview={FitH}}

\newtheorem{thm}{Theorem}[section]
\newtheorem{lem}[thm]{Lemma}

\newtheorem{conj}[thm]{Conjecture}
\theoremstyle{definition}

\numberwithin{equation}{section}
\newtheorem*{theorem}{Audenaert-Kittaneh's Conjecture}
\usepackage{mathrsfs}

\newcommand{\tr}{\text{tr }}

\begin{document}
	
	\title[Proof of Audenaert-Kittaneh's Conjecture]{Proof of Audenaert-Kittaneh's Conjecture}
	
	\author[T.~Zhang]{Teng Zhang}
	
	\address{School of Mathematics and Statistics, Xi'an Jiaotong University, Xi'an 710049, P. R. China}
	\email{teng.zhang@stu.xjtu.edu.cn}
	
	\subjclass[2020]{}
	% You may fill MSC codes, e.g.:
	% 26D10, 30A10, 15A60, 47A57.
	
	
	\begin{abstract}
		By using the Three-lines theorem for a certain analytic function defined in terms of the trace and a duality argument method, we prove Audenaert-Kittaneh's conjecture related to $p$-Schatten classes. This generalizes the main result obtained by McCarthy in \cite{Mc67}.
	\end{abstract}
	
	\maketitle
	%\tableofcontents
	
	\section{Introduction}
	
	The classical Clarkson's inequalities \cite{C36} for the Lebesgue spaces $L_p$, and their non-commutative analogues for the Schatten trace ideals, play an important role in analysis, operator theory, and mathematical physics. They have been generalised in various directions. Among these versions for more-general symmetric norms \cite{BH88}, for operators by $n$-tuples \cite{BK04,HK08} and for the Haagerup $L_p$ spaces \cite{FK86}, as well as refinements \cite{BCL94}. 
	More uniform convexity results appear in Pisier-Xu 's survey \cite[Chapter 1]{PX03}.
	
	Let $\mathbb{B}(\mathscr{H})$ denote the algebra of all bounded linear operators acting on a complex separable Hilbert space $\mathscr{H}$. If $X\in\mathbb{B}(\mathscr{H})$ is compact, we denote by $\{s_j(X)\}$ the sequence of decreasingly ordered singular values of $X$. For $p>0$, let
	\begin{eqnarray*}
		\|X\|_p=\big( \sum_{j}s_j^p(X)\big) ^{\frac{1}{p}}=\left(\tr \left|X\right|^p \right)^{\frac{1}{p}} ,
	\end{eqnarray*}
	where $\tr(\cdot)$  is the usual trace functional. This defines a norm (quasi-norm, resp.) for $1\le p<\infty$ ($0<p<1$, resp.) on the set 
	\[
	\mathbb{B}_p(\mathscr{H})=\{X\in \mathbb{B}(\mathscr{H}):\|X\|_p<\infty\},
	\]
	which is called the $p$-Schatten class of $\mathbb{B}(\mathscr{H})$; cf. \cite{GK69}.
	
	Clarkson's inequalities for operators $A$ and $B$ in $\mathbb{B}_p(\mathscr{H})$ (see \cite{Mc67}) assert that 
	\begin{thm}[McCarthy]
		Let $A,B\in \mathbb{B}_p(\mathscr{H})$. Then for $0<p\le 2$,
		\begin{eqnarray}\label{c1}
			2^{p-1}(\|A\|_p^p+\|B\|_p^p)\le \|A+B\|_p^p+\|A-B\|_p^p\le 	2(\|A\|_p^p+\|B\|_p^p),
		\end{eqnarray}
		and for $p\ge 2$,
		\begin{eqnarray}\label{c2}
			2(\|A\|_p^p+\|B\|_p^p)\le \|A+B\|_p^p+\|A-B\|_p^p\le 2^{p-1}(\|A\|_p^p+\|B\|_p^p).
		\end{eqnarray}
	\end{thm}
	
	For $p=2$ both inequalities \eqref{c1} and \eqref{c2} reduce to the parallelogram law
	\[
	\|A-B\|_2^2+\|A+B\|_2^2=2(\|A\|_2^2+\|B\|_2^2).
	\]
	
	One natural generalization of partial (1) and (2) is as follows:
	\begin{thm}[Hirazallah \& Kittaneh \cite{HK08}]
		Let $A_1,\ldots,A_n\in \mathbb{B}_p(\mathscr{H})$. Then for $0<p\le 2$,
		\begin{eqnarray}\label{c1n}
			n^{p-1}\sum_{i=1}^n\|A_i\|^p_p	\le \|\sum_{i=1}^nA_i\|_p^p+\sum_{1\le i<j\le n}\|A_i-A_j\|_p^p,
		\end{eqnarray}
		and for $p\ge 2$,
		\begin{eqnarray}\label{c2n}
			\|\sum_{i=1}^nA_i\|_p^p+\sum_{1\le i<j\le n}\|A_i-A_j\|_p^p\le n^{p-1}\sum_{i=1}^n\|A_i\|^p_p.
		\end{eqnarray}
	\end{thm}
	
	In 1994, Ball, Carlen and Lieb \cite{BCL94} explained the following Optimal 2-uniform convexity inequality.
	\begin{thm}[Ball, Carlen \& Lieb]
		Let $A,B\in \mathbb{B}_p(\mathscr{H})$. Then for $1\le p\le 2$,
		\begin{eqnarray}\label{ebcl}
			\left(\dfrac{\|A+B\|_p^p+\|A-B\|_p^p}{2}\right)^\frac{2}{p}\ge \|A\|_p^2+(p-1)\|B\|_p^2.
		\end{eqnarray}
		For $2\le p\le \infty$, the inequality is reversed.
	\end{thm}
	
	We note that the validity of inequality \eqref{ebcl} implies the  2-uniform convexity of both $L_p$ and $\mathbb{B}_p(\mathscr{H})$ for $1<p\le 2$. In contrast, Clarkson's inequality only implies that these spaces are $q$-uniformly convex over the same range $1<p\le 2$, see the following result.
	\begin{thm}[McCarthy]\label{Mc}
		Let $A,B\in \mathbb{B}_p(\mathscr{H})$. Then for $1<p\le 2$,
		\begin{eqnarray}\label{c3}
			\|A+B\|_p^q+\|A-B\|_p^q\le2(\|A\|_p^p+\|B\|_p^p)^\frac{q}{p},
		\end{eqnarray}
		and for $p\ge 2$,
		\begin{eqnarray}\label{c4}
			2(\|A\|_p^p+\|B\|_p^p)^\frac{q}{p}\le \|A+B\|_p^q+\|A-B\|_p^q
		\end{eqnarray}
		where $p,q>0$ and $\frac{1}{p}+\frac{1}{q}=1$.
	\end{thm}
	However, the proof of \eqref{c3} given by McCarthy collapses. See the remark at the end of paper by Fack-Kosaki in \cite{FK86}.
	
	In view of the inequalities \eqref{c3} and \eqref{c4}, it is reasonable to corresponding generalizations (along the lines of the inequalities \eqref{c1n} and \eqref{c2n} of the partial inequalities \eqref{c1} and \eqref{c2}) to $n$-tuples of operators. More precisely, in \cite[Section 8.1]{AK12} entitled ``Clarkson inequalities for several operators", the authors presented the following conjecture:
	\begin{theorem}\label{AKC}
		Let $A_1,\ldots,A_n\in \mathbb{B}_p(\mathscr{H})$. Then for $1<p\le 2$,
		\begin{eqnarray}\label{AK1}
			\left\| \sum_{i=1}^nA_i\right\| _p^q+\sum_{1\le i<j\le n}\|A_i-A_j\|_p^q\le n\left( \sum_{i=1}^n\|A_i\|^p_p\right)^\frac{q}{p},
		\end{eqnarray}
		and for $p\ge 2$,
		\begin{eqnarray}\label{AK2}
			n\left( \sum_{i=1}^n\|A_i\|^p_p\right)^\frac{q}{p}\le \left\| \sum_{i=1}^nA_i\right\| _p^q+\sum_{1\le i<j\le n}\|A_i-A_j\|_p^q,
		\end{eqnarray}
		where $p,q>0$ and $\frac{1}{p}+\frac{1}{q}=1$.
	\end{theorem}
	
	The above inequalities Audenaert and Kittaneh conjectured is a crucial piece of the Clarkson-type inequality puzzle. Notice that  both of \eqref{AK1} and \eqref{AK2} become equalities for $A_1=\cdots=A_n$, so these two inequalities are sharp. But little progress has been made in the decade since the conjecture was first proposed. It is worth mentioning that  Conde and Moslehian \cite{CM16} obtained a weaker version of Audenaert-Kittaneh's conjecture \ref{AKC} with a weaker coefficient $n^\frac{q}{2}$ instead of $n$ in 2016.
	\begin{thm}[Conde \& Moslehian]\label{CM}
		Let $A_1,\ldots,A_n\in \mathbb{B}_p(\mathscr{H})$. Then for $1<p\le 2$,
		\begin{eqnarray*}
			\left\| \sum_{i=1}^nA_i\right\| _p^q+\sum_{1\le i<j\le n}\|A_i-A_j\|_p^q\le n^\frac{q}{2}\left( \sum_{i=1}^n\|A_i\|^p_p\right)^\frac{q}{p},
		\end{eqnarray*} 
		and for $p\ge 2$,
		\begin{eqnarray*}
			n^\frac{q}{2}\left( \sum_{i=1}^n\|A_i\|^p_p\right)^\frac{q}{p}\le \left\| \sum_{i=1}^nA_i\right\|_p^q+\sum_{1\le i<j\le n}\|A_i-A_j\|_p^q,
		\end{eqnarray*}
		where $p,q>0$ and $\frac{1}{p}+\frac{1}{q}=1$.
	\end{thm}
	
	This paper is organized as follows: in Section \ref{s3}, by using Three-lines theorem for a certain analytic function defined in terms of the trace and a duality argument, we prove Audenaert-Kittaneh conjecture \ref{AKC}. It is a completely different path from Conde and Moslehian. In Section \ref{s4}, we present an open problem which is a possible generalization of Bourin and Lee in \cite[Theorem 4.1]{BL20}.
	
	\section{Proof of Audenaert-Kittaneh's conjecture}\label{s3}
	In this section, we give a proof of Audenaert-Kittaneh's conjecture \ref{AKC}.
	
	First, we introduced Hadamard three-lines theorem, which appears in the classic complex analysis textbook.
	\begin{lem}[Three-lines theorem, see \cite{U08}] \label{le1}
		Let $f(z)$ be a bounded function of $z=x+iy$ defined on the strip
		\[
		\{x+iy:a\le x\le b\},
		\]
		holomorphic in the interior of the strip and continuous on the whole strip. If 
		\[
		M(x):=\sup\limits_{y}\left|f(x+iy)\right|,
		\]
		then $\log M(x)$ is a convex function on $[a,b]$. In other words, if $x=ta+(1-t)b$ with $0\le t\le 1$, then
		\[
		M(x)\le M(a)^tM(b)^{1-t}.
		\]
	\end{lem}
	
	Next, we give an elegant use of Three-lines lemma \ref{le1} for a certain analytic function defined in terms of the trace, our proof is an extension of the proof in Fack-Kosaki in \cite[Theorem 5.3]{FK86}.
	\begin{lem}\label{leBK}
		Let $A_1, \ldots, A_n\in \mathbb{B}_p(\mathscr{H}) $ for some $1<p\le 2$. Denote $B=\sum\limits_{k=1}^nA_k, B_{i,j}=A_i-A_j, 1\le i<j\le n$. Let $Y, Y_{i,j}$ be operators in the dual class $\mathbb{B}_q(\mathscr{H})$. Then
		\begin{eqnarray}\label{e7}
			\left|\tr \left( YB+\sum_{1\le i<j\le n}Y_{i,j}B_{i,j}\right) \right|
			\le n^\frac{1}{q}\left(\sum_{k=1}^n \|A_k\|_p^p\right)^\frac{1}{p}
			\left(\left\|Y\right\|_q^p+\sum_{1\le i<j\le n}\|Y_{i,j}\|_q^p \right)^\frac{1}{p}.  
		\end{eqnarray} 
	\end{lem}
	
	\begin{proof}
		Let $A_k=\left|A_k\right|W_k,Y=V\left|Y\right|$ and $Y_{i,j}=V_{i,j}\left|Y_{i,j}\right|$ be right, left and left polar decompositions of $A_k, Y$ and $Y_{i,j}$, respectively. Here, $W_j,V$ and $V_{i,j}$ are partial isometries.
		
		We have $1/2\le 1/p<1$. For the complex variable $z=x+iy$ with $1/2\le x\le 1$, let
		\begin{eqnarray*}
			A_k(z)&=&\left|A_k\right|^{pz}W_k;\\
			B(z)&=&\sum\limits_{k=1}^nA_k(z);\\
			B_{i,j}(z)&=&A_i(z)-A_j(z);\\
			Y(z)&=&\|Y\|_q^{pz-q(1-z)}V\left|Y\right|^{q(1-z)};\\
			Y_{i,j}(z)&=&\|Y_{i,j}\|_q^{pz-q(1-z)}V_{i,j}\left|Y_{i,j}\right|^{q(1-z)}.
		\end{eqnarray*}
		Note that $A_k(1/p)=A_k, Y(1/p)=Y$ and $Y_{i,j}(1/p)=Y_{i,j}$. Let
		\[
		f(z)=\tr\left(Y(z)B(z)+ \sum_{1\le i < j<n} Y_{i,j}(z)B_{i,j}(z)\right) .
		\]
		The left-hand side of \eqref{e7} is $\left|f(1/p)\right|$. We can estimate this if we have bounds for $\left|f(z)\right|$ at $x=1/2$ and $x=1$. If $x=1$, we have
		\begin{eqnarray*}
			\left|\tr  Y(z)A_{k}(z) \right|
			&=&\|Y\|_q^p\left|\tr  V\left|Y\right|^{-iqy}\left|A_k\right|^{p(1+iy)}W_k \right|;\\
			\left|\tr  Y_{i,j}(z)A_{k}(z) \right|
			&=&\|Y_{i,j}\|_q^p\left|\tr  V_{i,j}\left|Y_{i,j}\right|^{-iqy}\left|A_k\right|^{p(1+iy)}W_k\right|.
		\end{eqnarray*}
		
		Using the information that for any operator $T$,
		\[
		\left|\tr T\right|\le \|T\|_1 \text{     and     } \|XTY\|\le \|X\|_\infty\|T\|\|X\|_\infty
		\]
		hold for three operators $X,T, Z$ and trace norm $\|\cdot\|$, spectral norm $\|\cdot\|_\infty$, any unitarily invariant norms $\|\cdot\|$, we see that
		\begin{eqnarray*}
			\left|\tr Y(z)A_{k}(z)\right|&\le& \|Y\|_q^p\left\|A_k\right\|_p^p;\\
			\left|\tr Y_{i,j}(z)A_{k}(z)\right|&\le& \|Y_{i,j}\|_q^p\left\|A_k\right\|_p^p.
		\end{eqnarray*}
		
		Hence
		\begin{eqnarray}
			\left|f(z)\right|
			&=&\left|\tr \left( Y(z)B(z)+\sum_{1\le i<j\le n}Y_{i,j}(z)B_{i,j}(z)\right) \right| \nonumber\\
			&=&\left|\tr \left( Y(z)\sum_{k=1}^nA_k(z)+\sum_{1\le i<j\le n}Y_{i,j}(z)\left( A_i(z)-A_j(z)\right) \right) \right| \nonumber\\
			&=&\left| \tr \sum_{k=1}^n\left( Y(z)-\sum\limits_{i=1}^{k-1}Y_{i,k}(z)+\sum\limits_{i=k+1}^nY_{k,i}(z)\right) A_k(z)\right|\nonumber \\
			&\le&\sum_{k=1}^n\left( \left| \tr  Y(z)A_k(z)\right|+\sum\limits_{i=1}^{k-1}\left| \tr Y_{i,k}(z)A_k(z)\right| +\sum\limits_{i=k+1}^n\left| \tr Y_{k,i}(z)A_k(z)\right|\right)   \nonumber \\
			&\le& \sum_{k=1}^n\left(\|Y\|_q^p+\sum\limits_{i=1}^{k-1}\|Y_{i,k}\|_q^p+\sum\limits_{i=k+1}^n\|Y_{k,i}\|_q^p\right)\|A_k\|_p^p   \nonumber \\
			&\le &\left(\left\|Y\right\|_q^p+\sum_{1\le i<j\le n}\|Y_{i,j}\|_q^p \right)\left(\sum_{k=1}^n \|A_k\|_p^p\right) \label{e8*}
		\end{eqnarray}
		when $x=1$.
		
		When $x=1/2$, the operators $A_k(z)$ and $Y_{i,j}(z)$ are in $\mathbb{B}_2(\mathscr{H})$ and 
		\begin{eqnarray*}
			\left|f(z)\right|
			&\le&\left|\tr Y(z)B(z)\right|+\sum_{1\le i<j\le n}\left|\tr Y_{i,j}(z)B_{i,j}(z)\right|\\
			&\le&\|Y(z)\|_2\|B(z)\|_2+\sum_{1\le i<j\le n}\|Y_{i,j}(z)\|_2\|B_{i,j}(z)\|_2\\
			&\le&\left(\left\|Y(z)\right\|_2^2+\sum_{1\le i<j\le n}\|Y_{i,j}(z)\|_2^2 \right)^\frac{1}{2}
			\left(\left\|B(z)\right\|_2^2+\sum_{1\le i<j\le n}\|B_{i,j}(z)\|_2^2 \right)^\frac{1}{2}\\
			&=& n^\frac{1}{2} \left(\left\|Y(z)\right\|_2^2+\sum_{1\le i<j\le n}\|Y_{i,j}(z)\|_2^2 \right)^\frac{1}{2}
			\left(\sum_{i=1}^n\|A_i(z)\|_2^2 \right)^\frac{1}{2}.  
		\end{eqnarray*}
		The equality at the last step is a consequence of Theorem \ref{CM}, specialised to the case $p=2$. Note that when $x=1/2$, we have $\|A_k(z)\|_2^2=\|A_k\|_p^p$, and $\|Y_{i,j}(z)\|_2^2=\|Y_{i,j}\|_q^p$. Hence
		\begin{eqnarray}\label{e8}
			\left|f(z)\right|\le n^\frac{1}{2} \left(\left\|Y\right\|_q^p+\sum_{1\le i<j\le n}\|Y_{i,j}\|_q^p \right)^\frac{1}{2}
			\left(\sum_{k=1}^n\|A_k\|_p^p \right)^\frac{1}{2}, 
		\end{eqnarray}
		when $x=1/2$. Let $M_1,M_2$ be the right sides of \eqref{e8*}, \eqref{e8}, repectively. Then by Lemma \ref{le1}, we have, for $1/2\le 1/p<1$,
		\[
		\left|f(1/p)\right|\le M_1^{2(1/p-1/2)}M_2^{2(1-1/p)}.
		\]
		This gives \eqref{e7}.
	\end{proof}
	
	Now, we give duality for Audenaert-Kittaneh's Conjecture by using a technique from Ball, Carlen and Lieb in \cite[Lemma 5 and Lemma 6]{BCL94}.
	\begin{thm}\label{AKd}
		(Duality for Audenaert-Kittaneh conjecture \ref{AKC})
		Let $1<p\le 2$ and $\dfrac{1}{p}+\dfrac{1}{q}=1$. The duality of
		\begin{eqnarray}\label{e2.4}
			\left\| \sum_{i=1}^nx_i\right\| _p^q+\sum_{1\le i<j\le n}\|x_i-x_j\|_p^q\le	n\left( \sum_{i=1}^n\|x_i\|^p_p\right)^\frac{q}{p}
		\end{eqnarray}
		for all $x_i\in \mathbb{B}_p(\mathscr{H})$ implies the validity of 
		\begin{eqnarray}\label{e2.3}
			n\left( \sum_{i=1}^n\|\phi_i\|^q_q\right)^\frac{p}{q}\le \left\| \sum_{i=1}^n\phi_i\right\| _q^p+\sum_{1\le i<j\le n}\|\phi_i-\phi_j\|_q^p
		\end{eqnarray}
		for all $\phi_i\in \mathbb{B}_q(\mathscr{H})$. 
	\end{thm}
	
	\begin{proof}
		Suppose \eqref{e2.4} holds in $\mathbb{B}_p(\mathscr{H})$, we try to establish \eqref{e2.3}. Let $\phi_i(1\le i\le n)\in \mathbb{B}_q(\mathscr{H})$. By Riesz representation theorem, we know there are unit vectors $\mu_i$ in the dual space of $\mathbb{B}_q(\mathscr{H})$ (i.e., $\mathbb{B}_p(\mathscr{H})$) such that
		\begin{eqnarray*}
			\mu_i(\phi_i)=\|\phi_i\|_q.
		\end{eqnarray*}
		Define $x_i\in \mathbb{B}_p(\mathscr{H})$ by
		\begin{eqnarray*}
			x_i=\left(\sum_{i=1}^n\|\phi_i\|^q_q\right)^{-\frac{1}{p}}\left\|\phi_i\right\|_q^{q-1}\mu_i.
		\end{eqnarray*}
		Then
		\begin{eqnarray}\label{eo1}
			\sum\limits_{i=1}^n\|x_i\|_p^p=1,
		\end{eqnarray}
		and we have
		\begin{eqnarray*}
			\left( \sum_{i=1}^n\|\phi_i\|^q_q\right)^\frac{1}{q}
			&=&\sum\limits_{i=1}^nx_i(\phi_i)\\
			&=&\dfrac{(\sum\limits_{i=1}^nx_i)(\sum\limits_{i=1}^n\phi_i)+\sum\limits_{1\le i < j\le n}(x_i-x_j)(\phi_i-\phi_j)}{n}\\
			&\le&\dfrac{\left\| \sum\limits_{i=1}^nx_i\right\|_p \left\| \sum\limits_{i=1}^n\phi_i\right\|_q +\sum\limits_{1\le i < j\le n}\left\| x_i-x_j\right\|_p \left\| \phi_i-\phi_j\right\|_q }{n}\\
			&&\text{ (By using triangle inequality)}\\
			&\le&\dfrac{\left(\left\| \sum\limits_{i=1}^nx_i\right\| _p^q+\sum\limits_{1\le i<j\le n}\|x_i-x_j\|_p^q\right)^\frac{1}{q}
				\left( \left\| \sum\limits_{i=1}^n\phi_i\right\| _q^p+\sum\limits_{1\le i<j\le n}\|\phi_i-\phi_j\|_q^p \right)^\frac{1}{p}  }{n}\\
			&&\text{ (By using H\"{o}lder inequality)}\\	
			&\le& \dfrac{n^\frac{1}{q}\left( \sum\limits_{i=1}^n\|x_i\|^p_p\right)^\frac{1}{p}
				\left( \left\| \sum\limits_{i=1}^n\phi_i\right\| _q^p+\sum\limits_{1\le i<j\le n}\|\phi_i-\phi_j\|_q^p\right) ^\frac{1}{p}}{n}
			\ \text{(By \eqref{e2.4})}\\ 
			&=& \dfrac{\left( \left\| \sum\limits_{i=1}^n\phi_i\right\| _q^p+\sum\limits_{1\le i<j\le n}\|\phi_i-\phi_j\|_q^p\right) ^\frac{1}{p}}{n^\frac{1}{p}}
			\ \text{(By \eqref{eo1})}.
		\end{eqnarray*}
		That is, \eqref{e2.4} implies \eqref{e2.3}.
	\end{proof}
	
	\noindent\emph{Proof of Audenaert-Kittaneh's Conjecture \ref{AKC}: }
	Now, let $B=U\left|B\right|,B_{i,j}=U_{i,j}\left|B_{i,j}\right|$ be polar decompositions, and let
	\[
	Y=\left\|B\right\|_p^{q-p}\left|B\right|^{p-1}U^*,\qquad
	Y_{i,j}=\left\|B_{i,j}\right\|_p^{q-p}\left|B_{i,j}\right|^{p-1}U_{i,j}^*.
	\]
	It is easy to see that
	\[
	\tr YB=\left\|B\right\|_p^q=\left\|Y\right\|_q^p,\qquad
	\tr Y_{i,j}B_{i,j}=\left\|B_{i,j}\right\|_p^q=\left\|Y_{i,j}\right\|_q^p.
	\]
	So from Lemma \ref{leBK}, we get 
	\[
	\|B\|_p^q+\sum_{1\le i<j\le n}\|B_{i,j}\|_p^q\le n^\frac{1}{q}\left(\sum_{i=1}^n\left\|A_i\right\|_p^p \right)^{1/p}
	\left(\|B\|_p^q+\sum_{1\le i<j\le n}\|B_{i,j}\|_p^q\right)^{1/p}.
	\]
	This is the same as saying that
	\[
	\left\| \sum_{i=1}^nA_i\right\| _p^q+\sum_{1\le i<j\le n}\|A_i-A_j\|_p^q\le n\left( \sum_{i=1}^n\|A_i\|^p_p\right)^\frac{q}{p},\qquad 1<p\le 2.
	\]
	This proves the conjecture for $1<p\le 2$. The reverse inequality for $2\le p<\infty$ can be obtained from Theorem \ref{AKd}. \qed
	
	\section{Further remarks}\label{s4}
	There exists a recent considerable improvement of \eqref{c2} in the matrix setting due to J.-C. Bourin and E.-Y. Lee \cite[Theorem 4.1]{BL20}. 
	\begin{thm}\label{BLm}
		Let $A,B$ be two $m\times m$ matrices and $p>2$. Then there exists two unitary matrices $U,V$ such that
		\[
		U\left|A+B\right|^pU^*+V\left|A-B\right|^pV^*\le 2^{p-1}\left(\left|A\right|^p+\left|B\right|^p \right). 
		\]
		For $0<p\le 2$, the inequality is reversed.
	\end{thm}
	
	In view of \eqref{c2n}, we conjecture a $n$-variable version of Theorem \ref{BLm}.
	\begin{conj}
		Let $A_1,\ldots,A_n$ be $m\times m$ matrices and $p>2$. Then there exists $\tfrac{n(n-1)}{2}+1$ unitary matrices $U,U_{i,j}\ (1\le i<j\le n)$ such that
		\[
		U\left|\sum\limits_{i=1}^n A_i \right|^pU^*+\sum_{1\le i<j\le n}U_{i,j}\left|A_i-A_j\right|^pU_{i,j}^*\le n^{p-1}\sum_{i=1}^n\left|A_i\right|^p.
		\]
		For $0<p\le 2$, the inequality is reversed.
	\end{conj}
	There are still many open problems on the famous Clarkson-McCarthy's inequalities, see \cite{PX03}.
	
	\section*{Declaration of competing interest}
	The author declares no competing interests.
	
	\section*{Data availability}
	No data was used for the research described in the article.
	
	\section*{Acknowledgments}
	Teng Zhang is supported by the China Scholarship Council, the Young Elite Scientists Sponsorship Program for PhD Students (China Association for Science and Technology), and the Fundamental Research Funds for the Central Universities at Xi'an Jiaotong University (Grant No.~xzy022024045). 
	The author is deeply grateful to Professor Eric A. Carlen for reviewing an earlier version of this manuscript and offering valuable suggestions. Thanks are also extended to Professor M. Lin and the referee for their insightful comments.

\end{document}